\documentclass[11pt]{amsart}

\usepackage[utf8]{inputenc}
\usepackage{a4wide}
\usepackage{color}
\usepackage{amsfonts,amsmath,amsthm,amssymb}
\usepackage{verbatim}
\usepackage{epsfig}
\usepackage{graphics}
\usepackage{mathtools}
\usepackage{array}
\usepackage{enumerate}
\usepackage[ruled]{algorithm}
\usepackage{algorithmic}
\usepackage{bbm}
\usepackage{caption}
\usepackage{subcaption}
\usepackage[all, knot, cmtip]{xy} 
\usepackage{wrapfig}
\usepackage{marginnote}
\usepackage{caption}

\newcommand{\ndN}{\mathbb{N}}

\newcommand{\ndR}{\mathbb{R}}

\newcommand{\ndK}{\mathbb{K}}

\renewcommand{\Pr}[1]{\mathbb{P}\left(#1\right)}
\newcommand{\Ex}[1]{\mathbb{E}[#1]}

\newcommand{\cA}{\mathcal{A}}

\newcommand{\cC}{\mathcal{C}}
\newcommand{\cD}{\mathcal{D}}
\newcommand{\cE}{\mathcal{E}}

\newcommand{\cS}{\mathcal{S}}
\newcommand{\cT}{\mathcal{T}}

\newcommand{\cX}{\mathcal{X}}

\newcommand{\mA}{\mathsf{A}}

\newcommand{\me}{\mathsf{e}}

\newcommand{\CRT}{\mathcal{T}_\me}

\newcommand{\He}{\textsc{H}}
\newcommand{\Di}{\textsc{D}}
\newcommand{\Wi}{\textsc{W}}

\newcommand{\eqdist}{\,{\buildrel (d) \over =}\,}

\newcommand{\convdis}{\,{\buildrel (d) \over \longrightarrow}\,}

\newcommand{\disto}{\mathsf{dis}}

\newcommand{\MSet}{\textsc{MSET}}

\newtheorem{theorem}{Theorem}[section]

\newtheorem{proposition}[theorem]{Proposition}
\newtheorem{lemma}[theorem]{Lemma}

\newtheorem{definition}[theorem]{Definition}

\numberwithin{equation}{section}

\title{Scaling Limits of Random P\'{o}lya Trees}

\author{Konstantinos Panagiotou}
\address[Konstantinos Panagiotou]{Institute of Mathematics, Ludwig-Maximilians-Universit\"at, Theresienstr.\ 39, 80333 Munich, Germany}
\email{kpanagio@math.lmu.de}

\author{Benedikt Stufler}
\address[Benedikt Stufler]{Institute of Mathematics, Ludwig-Maximilians-Universit\"at, Theresienstr.\ 39, 80333 Munich, Germany}
\email{stufler@math.lmu.de}

\date{\today}

\begin{document}

\begin{abstract}
P\'olya trees are rooted  trees considered up to symmetry. We establish the convergence of large uniform random P\'olya trees with arbitrary degree restrictions to Aldous' Continuum Random Tree with respect to the Gromov-Hausdorff metric. Our proof is short and elementary, and it shows that the global shape of a  random P\'olya tree is essentially dictated by a large Galton-Watson tree that it contains. We also derive sub-Gaussian tail bounds for both the height and the width, which  are optimal up to constant factors in the exponent. 
\end{abstract}

\maketitle

\section{Introduction and main results}

Any connected graph $G$ with vertex set $V$ can be associated in a natural way with a metric space $(V, d_G)$, where $d_G(u,v)$ is defined as the length of a shortest path that contains $u$ and $v$ in $G$. In this paper we consider the setting where $G$ is a random tree with $n$ vertices, and we study, as $n \to \infty$, several properties of the associated random metric space.

The most prominent and well-studied case that fits in our setting is when $G$ is a critical Galton-Watson random tree with $n$ vertices, where the offspring distribution has a finite non-zero variance. In the series \cite{MR1085326, MR1166406, MR1207226} of seminal papers Aldous proved that the metric spaces associated to those trees admit a common and universal limit, the so-called Continuum Random Tree (CRT). Since then, the CRT has been shown to be the limit of various  families of random combinatorial structures, in particular other distributions on trees, see e.g.\ Haas and Miermont \cite{MR3050512} and references therein, planar maps, see e.g.\ Albenque and Marckert \cite{MR2438817}, Bettinelli \cite{MR3335010}, Caraceni \cite{Ca}, Curien, Haas and Kortchemski \cite{RSA20554}, Janson and Stefansson \cite{MR3342658}, Stufler \cite{2015arXiv150507600S}, and certain families of graphs, see Panagiotou, Stufler, and Weller \cite{2014arXiv1411.1865P}.

Here we study the class of
\emph{P\'olya trees}, which are rooted trees (that is, there is a distinguished vertex called the root) considered up to symmetry, equipped with the uniform distribution. They are named after George P\'olya, who developed a framework based on generating functions in order to study their properties~\cite{MR1577579}. The study of these objects, especially in random settings, poses significant difficulties: the presence of non-trivial symmetries makes it difficult to derive an explicit and handy description of the probability space at hand. To wit, random P\'olya trees do not fit into well-studied models of random trees such as Galton-Watson trees, a fact that was widely believed and which was established rigorously  by Drmota and Gittenberger \cite{MR2718280}.

The main contribution of this paper is a simple and short proof that establishes  the scaling limit of random P\'olya trees, with the additional benefit that it allows us to consider \emph{arbitrary} degree restrictions. That is, we may restrict the outdegrees of the vertices to an arbitrary set (always including, of course, 0 and an interger $\ge 2$, so that the trees are finite and non-trivial). Our proof also reveals a novel striking structural property that is of independent interest and that rectifies the common perception of random unlabelled trees. As already mentioned, it is known that random P\'olya trees do not admit a simple probabilistic description. However, we argue that this barely fails to be the case, namely that a random P\'olya tree ``consists'' in a well-defined sense of a large Galton-Watson tree having a random size to which small forests are attached -- in other words, the global structure of a large random P\'olya tree is similar to the structure of the Galton-Watson tree it contains.
\begin{theorem}
\label{te:main}
Let~$\Omega$ be an aribtrary set of nonnegative integers containing zero and at least one integer greater than or equal to two. Let~$\mA_n$ denote the uniform random P\'olya tree with~$n$ vertices and vertex outdegrees in~$\Omega$.  Then there exists a constant~$c_\Omega>0$ such that the metric space
	$(\mA_n, c_\Omega n^{-1/2} d_{\mA_n})$ 
converges towards the continuum random tree	$(\CRT, d_{\CRT})$ in the Gromov-Hausdorff sense as $n \equiv 1 \mod \gcd(\Omega)$ tends to infinity. 
\end{theorem}
In the theorem we use the normalization of Le Gall \cite{MR2225746} and let $\CRT$ denote the continuum random tree constructed from Brownian excursion, see Section~\ref{sec:ald} for the appropriate definitions. We also obtain explicit expressions for the scaling constant $c_\Omega$ in our proof, see~\eqref{SePoeq:cOmega} and the subsequent equations.

Random P\'olya trees were studied in several papers prior to this work. In particular, since the construction of the CRT in the early 90's it was a long-standing conjecture \cite[p.\ 55]{MR1166406} that this model of random trees (without any degree restrictions) also allows the same scaling limit. The convergence of binary P\'olya trees, that is, when the vertex outdegrees are restricted to the set $\{0,2\}$, was established by Marckert and Miermont \cite{MR2829313} using an appropriate trimming procedure. Later, in \cite{MR3050512} the conjecture was proven by using different techniques; actually, a far more general result on the scaling limit of random trees satisfying a certain Markov branching property was shown. Among other results, the method in~\cite{MR3050512} allows also to study P\'olya trees with some degree restrictions, where the vertex outdegrees have to be constrained in a set of the form $\{0, 1, \dots, d\}$ or $\{0, d\}$ for $d \ge 2$. However, the question about the convergence of P\'olya trees with \emph{arbitrary} degree restrictions was open, and we answer it in this work with a simple argument.

Our next result is concerned with two extremal parameters. The  height $\He(T)$ of a rooted tree $T$ is defined as the maximal distance of a vertex from the root, and the width $\Wi(T)$ is the maximal number of vertices at any fixed distance from the root. 
\begin{theorem}
\label{te:main2}
Let~$\Omega$ be an aribtrary set of nonnegative integers containing zero and at least one integer greater than or equal to two. Let~$\mA_n$ denote the uniform random P\'olya tree with~$n$ vertices and vertex outdegrees in~$\Omega$. Then there are constants $C,c>0$ such that
\[
\Pr{\He(\mA_n) \ge x} \le C \exp(-cx^2/n), \quad \Pr{\Wi(\mA_n) \ge x} \le C \exp(-cx^2/n)
\]
for all $x \ge 0$ and $n \equiv 1 \mod \gcd(\Omega)$.
\end{theorem}
Similar bounds were obtained by Addario-Berry, Devroye and Janson \cite{MR3077536} for critical Galton-Watson trees with finite nonzero variance, conditioned to be large. Our proofs show that these bounds are (up to the choice of $c,C$) best possible. As a direct consequence of our results we obtain for the distribution of $\He$ that
\[
c_{\Omega} n^{-1/2} \He(\mA_n) \convdis \He(\CRT)
\qquad
\textrm{and}
\qquad
\Ex{\He(\mA_n)^p} \sim c_{\Omega}^{-p} n^{p/2} \Ex{\He(\CRT)^p}
\]
for all $p \ge 1$. 
The distribution of the height $\He(\CRT)$ is known and given by
\begin{align}
\label{eq:height}
\He(\CRT) \eqdist \sup_{0 \le t_1 \le t_2 \le 1} \me(t),
\end{align}
where $\me(t)$ is a Brownian excursion of duration one, and
\begin{align}
\label{eq:height1}
\Pr{\He(\CRT) > x} =  2 \sum_{k \ge 1} (4k^2x^2-1)\exp(-2k^2x^2),
\end{align}
see \cite[Ch.\ 3.1]{MR1166406}. Its moments are also known and given by
\begin{align*}
\Ex{\He(\CRT)} &= \sqrt{\pi / 2}, \quad\,\,\,\,  \Ex{\He(\CRT)^p} = 2^{-p/2} p(p-1) \Gamma(p/2) \zeta(p) \quad \text{for $p \ge 2$}.
\end{align*}
This follows from standard results for Brownian excursion by Chung \cite{MR0467948}, or by results of R\'enyi and Szekeres \cite[Eq.\ (4.5)]{MR0219440}) who calculated the moments of the limit distribution of the height of a class of trees that converges towards the CRT; see also~\cite{2014arXiv1411.1865P}.

\subsection*{Methods} Our method relies on generating random P\'olya trees using the framework of Boltzmann samplers \cite{MR2498128, MR2810913}. Our main insight is that this allows us to show that with high probability, that is, with probability tending to one as~$n\to\infty$, the  shape of the P\'olya tree~$\mA_n$ is given by a subtree~$\cT_n$ with small subtrees that contain~$O(\log n )$ vertices attached to each vertex. As a metric space, we argue that~$\cT_n$ is distributed like a critical Galton-Watson tree, whose offspring distribution even has finite exponential moments, conditioned on having a randomly drawn size concentrating around~$n$ times a constant. In particular, by using Aldous's fundamental result~\cite{MR1207226}, see also~\cite{MR2603061,MR3025391}, we obtain that~$\cT_n$ converges to a multiple of the CRT; moreover, the Gromov-Hausdorff distance of~$(\mA_n, n^{-1/2}d_{\mA_n})$ and~$(\cT_n, n^{-1/2}d_{\cT_n})$ converges in probability to zero, yielding the desired result. To prove Theorem~\ref{te:main2} we then use tail-bounds for $\He(\cT_n)$ in order to obtain the corresponding bounds for the height of $\mA_n$.

\subsection*{Outline} The paper is structured as follows. In the next section we recall Aldous' theorem regarding the convergence of Galton-Watson trees, and we introduce some notation that will be used throughout the paper. Since this paper is targeted to a probabilistic audience, we introduce in Section~\ref{sec:ccgs} all required combinatorial preliminaries, tailored to our specific aims. In Section~\ref{sec:poltree} we give a formal definition of P\'olya trees and derive the sampling algorithm that will be the basis of our analysis. Finally, in Section~\ref{sec:pfmain}, which is the main novel contribution of this paper, we present the proofs of our main theorems.

\section{Aldous' fundamental theorem}  \label{sec:ald}
\subsection{Gromov-Hausdorff convergence}
The exposition here is based on \cite[Ch.\ 7]{MR1835418} and \cite{MR3025391}. A {\em pointed} metric space is a metric space together with a distinguished element that is often referred to as its root. A {\em correspondence} between two pointed metric spaces $X^\bullet = (X,d_X, x_0)$ and $Y^\bullet =(Y,d_Y,y_0)$ is a subset $R \subset X \times Y$ such that $(x_0, y_0) \in R$, and for each $x \in X$ there is a point $y \in Y$ with $(x,y) \in R$, and for each $y \in Y$ there is a point $x \in X$ with $(x,y) \in R$. The {\em distortion} of the correspondence $R$ is given by
\[
\disto(R) = \sup_{(x_1, y_1), (x_2, y_2) \in R} |d_X(x_1,x_2) - d_Y(y_1,y_2)|.
\]
If $(X,d_X)$ and $(Y,d_Y)$ are compact, the {\em Gromov-Hausdorff distance} between $X^\bullet$ and $Y^\bullet$ is defined by
\[
d_{\text{GH}}(X^\bullet,Y^\bullet) = \frac{1}{2} \inf_R \disto(R) \in [0, \infty[,
\]
with the index $R$ ranging over all correspondences between $X^\bullet$ and $Y^\bullet$, see \cite[Thm.\ 7.3.25]{MR1835418} and \cite[Prop.\ 3.6]{MR3025391}.  

Two pointed metric spaces are \emph{isometric}, if there exists a distance preserving bijection between the two that also preserves the roots. The Gromov-Hausdorff distance does not change if one of the spaces is replaced by another isometric copy. Moreover, two pointed spaces have Gromov-Hausdorff distance zero, if and only if they are isometric, and the Gromov-Hausdorff distance satisfies the axioms of a premetric on the collection of compact pointed metric spaces, see \cite[Thm.\ 7.3.30]{MR1835418} and \cite[Thm.\ 3.5]{MR3025391}. We may thus view $d_{\text{GH}}$ as a metric on the collection $\ndK^\bullet$ of all isometry classes of compact pointed metric spaces. 
\subsection{The continuum random tree}
The continuum random tree (CRT) $\CRT$ is a random metric space that is encoded by the Brownian excursion of duration one. We briefly introduce it following \cite{MR3025391,MR2203728}. Given an arbitrary continuous function $f: [0,1] \to [0, \infty[$  satisfying $f(0) = f(1) = 0$ we may define a premetric $d$ on the interval $[0,1]$ given by
\[
d(u,v) = f(u) + f(v) - 2 \inf_{u \le s \le v} f(s)
\]
for $u \le v$. Let $(\cT_f, d_{\cT_f}) = ([0,1]/ \mathord{\sim}, \bar{d})$ denote the corresponding quotient space obtained by identifying points that have distance zero. We   consider this space as being rooted at the equivalence class $\bar{0}$ of $0$. The random pointed metric space $(\CRT, d_{\CRT}, \bar{0})$ coded by the Brownian excursion of duration one $\me=(\me_t)_{0 \le t \le 1}$ is called the \emph{Brownian continuum random tree} (CRT).

\subsection{Plane trees and Aldous' theorem}

The \emph{Ulam-Harris tree} is defined as an infinite rooted tree with vertex set $\cup_{n\in \mathbb{N}_0} \mathbb{N}^n$ consisting of finite
sequences of natural numbers. The empty string $\emptyset$ is its root,  and the offspring of any vertex $v$ is given by the concatenations $\{vi: i \in \mathbb{N}\}$. In particular, the labelling of the vertices induces a linear order on each offspring set. A \emph{plane tree} is defined as a subtree of the Ulam-Harris tree that contains the root. Any plane tree is a pointed metric space with respect to the graph-metric and the root vertex $\emptyset$. Hence random plane trees may be considered as random elements of the metric space $\ndK^\bullet$.

Let $\xi$ be a random variable with support on $\mathbb{N}_0$. Then, a $\xi$-Galton-Watson tree $\mathcal{T}$ is the family tree of a Galton-Watson branching process with offspring distribution $\xi$, interpreted as a (possibly infinite) plane tree. We call $\mathcal{T}$ \emph{critical} if $\mathbb{E}[\xi] = 1$. The following invariance principle giving a scaling limit for certain random plane trees is due to Aldous \cite{MR1207226} and there exist various extensions, see for example \cite{MR1964956, MR2147221,MR3050512}.
\begin{theorem}
\label{te:gwtconv}
Let $\cT_n$ be a critical $\xi$-Galton-Watson tree conditioned on having $n$ vertices, with the offspring distribution $\xi$ having finite non-zero variance $\sigma^2$. As $n$ tends to infinity,  $\cT_n$ with edges rescaled to length $\frac{\sigma}{2 \sqrt{n}}$ converges in distribution to the CRT, that is
\begin{equation*}
	(\cT_n, \frac{\sigma}{2 \sqrt{n}} d_{\cT_n}, \emptyset ) \convdis (\CRT, d_{\CRT}, \bar{0})
\end{equation*}
in the metric space $(\ndK^\bullet, d_{\text{GH}})$.
\end{theorem}
In the following we  use a more compact notation, writing $\alpha \cT_n$ and $\CRT$ when refering to $(\cT_n, \alpha d_{\cT_n}, \emptyset )$ and $(\CRT, d_{\CRT}, \bar{0})$.

\subsection{Tail-bounds for the height and width}
In \cite[Thm.\ 1.2]{MR3077536} the following tail-bounds were obtained.
\begin{theorem}
\label{te:theoconv}
Let $\cT_n$ be a critical $\xi$-Galton-Watson tree conditioned on having $n$ vertices, with the offspring distribution $\xi$ having finite non-zero variance $\sigma^2$. Then there are constants $C,c>0$ such that for all $x\ge0$ and $n$
\[
\Pr{\He(\cT_n) \ge x} \le C \exp(-cx^2/n), \quad \Pr{\Wi(\cT_n) \ge x} \le C \exp(-cx^2/n).
\]
\end{theorem}

\section{Combinatorial preliminaries} \label{sec:ccgs}

We recall relevant notions and tools from combinatorics. In particular, we discuss constructions of combinatorial classes following Joyal \cite{MR633783} and Flajolet and Sedgewick \cite{MR2483235}, and give a brief account on Boltzmann samplers following Flajolet, Fusy, and Pivoteau \cite{MR2498128}. This will be our main tools for studying the class of P\'olya trees.

\subsection{Combinatorial classes and generating series}

A {\em combinatorial class} $\cC$ is a set together with a {\em size-function} $| \cdot |: \cC \to \ndN_0$. We require that for any $n \in \mathbb{N}_0$ the subset $\cC_n \subset \cC$ of all $n$-sized elements is finite. The {\em ordinary generating series} of a class $\cC$ is defined as the formal power series
\[
\cC(z) = \sum_{n \in \mathbb{N}_0} |\cC_n| z^n,
\]
with $|\cC_n|$ denoting the number of elements of the set $\cC_n$. We set $[z^n]\cC(z) = |\cC_n|$.

\subsection{Permutations}
As an example of a combinatorial class, consider \[\cS = \bigcup_{n \in \mathbb{N}_0} \cS_n\] of all permutations with $\cS_n$ denoting the symmetric group of order $n$. Its ordinary generating series is given by \[\cS(z) = \sum_{n \in \mathbb{N}_0} n! z^n.\]
Recall that any permutation $\sigma$ may be written in an essentially unique way as a product of disjoint cycles (corresponding to the orbits of the permutation). In the following we are going to let $\sigma_i$, $i \ge 1$ denote the number of cycles of length $i$, that is, with exactly $i$ elements, in this factorization. Here we count fixpoints as $1$-cycles.

\subsection{Operations on classes}
\subsubsection{Product classes}
Given two classes $\cC$ and $\cD$ we may form the {\em product class} as the set-theoretic product \[\cC \cdot \cD = \cC \times \cD\] with the size-function given by $|(C,D)| = |C| + |D|$ for any $C \in \cC$ and $D \in \cD$. It is a straightforward consequence, see also  Chapter I.1.\ in~\cite{MR2483235}, that
\[
(\cC \cdot \cD)(z) = \cC(z) \cD(z).
\]

\subsubsection{Multisets}
We may also form the class 
$\MSet(\cC)$ of all {\em multisets} of elements of $\cC$, that is, sets of the form
\[
\{(C_1, n_1), \ldots, (C_k, n_k)\}
\]
with $k \in \ndN_0$, $C_1, \ldots, C_k \in \cC$ being pairwise distinct and $n_1, \ldots, n_k \in \ndN$. Here the sum $\sum_{i=1}^n n_i$ denotes the {\em number of elements} of the multiset. The size-function for multisets is given by
\[
|\{(C_1, n_1), \ldots, (C_k, n_k)\}| = \sum_{i=1}^k |C_i|n_i.
\]
For any subset $\Omega \subset \ndN_0$ we may also form the class $\MSet_\Omega(\cC) \subset \MSet(\cC)$ by restricting to multisets whose number of elements lies in $\Omega$.

In order to express the ordinary generating series for the class of multisets in $\cC$, we require the concept of {\em cycle index sums}. Recall that for any permutation $\sigma$ we let $\sigma_i$ denote the number of cycles of length $i$ of $\sigma$. 

\begin{definition}
For any subset $\Omega \subset \ndN_0$ define the cycle index sum
\[
Z_\Omega(s_1, s_2, \ldots) = \sum_{k \in \Omega} \frac{1}{k!} \sum_{\sigma \in \cS_k} s_1^{\sigma_1} \cdots s_k^{\sigma_k}.
\]
\end{definition}
\noindent
For example, when $\Omega=\ndN_0$ a short calculation, see below, shows that 
\[
Z_{\ndN_0}(s_1, s_2, \dots) = \exp\Big( \sum_{i \ge 1} {s_i}/{i}\Big).
\]
Indeed, for any permutation $\sigma$ the series $(\sigma_i)_{i \in \mathbb{N}}$ is an element of the   space $\ndN_0^{(\ndN)}$ of all sequences in $\ndN_0$ with finite support. Conversely, to any element $m=(m_i)_{i \in \mathbb{N}} \in \ndN_0^\ndN$ correspond only permutations of order $n = \sum_{i=1}^{\infty} i m_i$ and their number is given by $n! / \prod_{i \ge 1} ( m_i! \, i^{m_i})$. Hence 
\begin{align*}
Z_{\ndN_0}  = \sum_{m \in \ndN_0^{(\ndN)}} \prod_{i \ge 1}  \frac{s_i^{m_i}} { m_i! \, i^{m_i}} = \prod_{i \ge 1} \sum_{m_i \ge 0}   \frac{s_i^{m_i}} { m_i! \, i^{m_i}} = \prod_{i \ge 1}  \exp ({s_i}/{i}) = \exp( \sum_{i \ge 1} {s_i}/{i}).
\end{align*}
We may now express the ordinary generating series for a multiset of objects. This result is implicit in Harary and Palmer~\cite{MR0357214} as an application of P\'olya's Enumeration Theorem; see also Joyal \cite[Prop. 9]{MR633783} for a formulation in a more general setting.
\begin{proposition}
\label{pro:numset}
The ordinary generating series of a multiset class $\MSet_\Omega(\cC)$ is given by
\[
\MSet_\Omega(\cC)(z) = Z_\Omega(\cC(z), \cC(z^2), \cC(z^3), \ldots).
\]
In particular,
\[
\MSet(\cC)(z) = \exp\big(\sum_{k \ge 1} \cC(z^k)/k\big).
\]
\end{proposition}

\subsection{Boltzmann samplers}
\label{se:bosa}

Given a nonempty combinatorial class $\cC$ and a parameter $x>0$ with $\cC(x) < \infty$, we may consider the corresponding {\em Boltzmann distribution} on $\cC$ that assigns probability weight $x^{|C|} / \cC(x)$ to any element $C \in \cC$. A {\em Boltzmann sampler} $\Gamma \cC (x)$ is a stochastic process that generates elements from $\cC$ according to the Boltzmann distribution  with parameter $x$. There are various rules according to which we may construct such samplers.

\subsubsection{Product}
Let $\cC$ and $\cD$ be nonempty combinatorial classes and $x>0$ such that $\cC(x), \cD(x)$ are finite. Then a Boltzmann sampler $\Gamma (\cC \cdot \cD)(x)$ for the product is given as follows.
\begin{enumerate}[\qquad 1.]
	\item Draw an element $C \in \cC$ using a Boltzmann sampler $\Gamma \cC(x)$.
	\item Draw independently an element $D \in \cD$ using a Boltzmann sampler $\Gamma \cD(x)$.
	\item Return the pair $(C,D)$.
\end{enumerate}
See, for example, in Section 2 of \cite{MR2498128} for the (simple) justification.

\subsubsection{Multiset}
Let $\cC$ be a nonempty combinatorial class and $\Omega \subset \ndN_0$ a subset. Then for any parameter $x>0$ with $\MSet_\Omega(\cC)(x)< \infty$ a Boltzmann sampler $\Gamma \MSet_\Omega(\cC)(x)$ is given as follows.
\begin{enumerate}[\qquad 1.]
	\item Draw a permutation $\sigma$ from   $\bigcup_{k \in \Omega}  \cS_k$ such that for each $k \in \Omega$ and $\nu \in \cS_k$
	\[
		\Pr{\sigma = \nu} = \frac{1}{k!} \cC(x)^{\nu_1} \cC(x^2)^{\nu_2} \cdots \cC(x^k)^{\nu_k} / \,\MSet_\Omega(\cC)(x).
	\]
	\item For each cycle $\tau$ of $\sigma$ let $|\tau|$ denote its length. Draw a random graph $C_\tau$ using a Boltzmann sampler $\Gamma \cC(x^{|\tau|})$.
	\item Return the multiset of $\cC$-objects that contains any $C \in \cC$ precisely $\sum_{\tau, C_\tau = C} |\tau|$ times, with the index $\tau$ ranging over all cycles of $\sigma$.
\end{enumerate}
For a proof see \cite[Prop.\ 38]{MR2810913} and \cite[Thm. 4.2]{MR2498128}. 

\section{Random P\'olya trees} \label{sec:poltree}

\subsection{Combinatorial decomposition of P\'olya trees}
Let $\Omega \subset \ndN_0$ be a subset containing $0$ and at least one integer $\ge 2$. Let $\cA_\Omega$ denote the combinatorial class of P\'olya trees with vertex outdegrees in $\Omega$.
Any P\'olya tree $A$ is uniquely determined by the multiset \[
M(A) = \{(A_1, n_1), \ldots, (A_k, n_k)\}
\]
of smaller P\'olya trees obtained by removing the root vertex of $A$. The tree $A$ has vertex outdegrees in $\Omega$ if and only if the number of elements of the multiset lies in $\Omega$, and if each of its elements $A_i$ belongs to $\cA_\Omega$. Thus, letting $\cX = \{o\}$ denote the combinatorial class constisting of a single object with size $1$, the map
\begin{align}
\label{eq:map}
\cA_\Omega \to \cX \cdot \MSet(\cA_\Omega), \quad A \mapsto (o, M(A))
\end{align}
is a size-preserving bijection. Using Proposition \ref{pro:numset}, this yields the equation
\begin{align}
\label{eq:relation}
\cA_\Omega(z) = z Z_\Omega(\cA_\Omega(z),\cA_\Omega(z^2), \ldots ).
\end{align}

\subsection{Enumerative properties}
\label{SeFrsec:enumprop}
In this section we collect basic analytic facts regarding P\'olya trees, which are frequently used in the proofs of the main theorems.
The following result is obtained by applying a general enumeration theorem due to Bell, Burris and Yeats \cite[Thm.\ 75]{MR2240769}. Special cases such as for trees with less general vertex-degree restrictions are classical combinatorial results, see e.g.~\cite[Thm. VII.4]{MR2483235} but also P\'olya \cite{MR1577579} and Otter \cite{MR0025715}. We do provide an explicit proof for the readers convenience, but do not claim novelty of this result. Although it does not seem to be  explicitly stated in this generality in the literature, it is implicit in the work \cite{MR2240769} and the present proof summarizes the corresponding arguments.

\begin{proposition}
	\label{pro:enumeration}
	\label{SeFrpro:ser1}
	\label{SeFrpro:cnt}
	Let~$\rho_\Omega$ denote the radius of convergence of the ordinary generating function~$\cA_\Omega(z)$. Then the following holds.
	\begin{enumerate}[i)]
		\item We have that~$0<\rho_\Omega<1$ and~$0< \cA_\Omega(\rho_\Omega) < \infty$.
		\item For some~$\epsilon > 0$, the function~$E(z,w) = zZ_\Omega(w,\cA_\Omega(z^2), \cA_\Omega(z^3), \ldots)$ satisfies \[E(\rho_\Omega + \epsilon, \cA_\Omega(\rho_\Omega) + \epsilon)< \infty.\]
		\item For some constant~$d_\Omega > 0$, the number of P\'olya trees with~$n$ vertices and outdegrees in ~$\Omega$ is given by \[[z^n]\cA_\Omega(z) \sim d_\Omega n^{-3/2} \rho_\Omega^{-n}.\] 
	\end{enumerate}
\end{proposition}
\begin{proof}
	We start with the proof of i).
	The series $\cA_{\Omega}(z)$ is dominated coefficentwise by the ordinary generating series $\cA(z) = \cA_{\ndN_0}(z)$ of all P\'olya trees  and it is known that $\cA(z)$ is analytic at the origin (see e.g.\ \cite[Prop. VII.5]{MR2483235} and \cite{MR1577579,MR0025715}).
	Hence ${\rho_\Omega}>0$.  As formal power series we have by  \eqref{eq:relation} that $\cA_{\Omega}(z) = z Z_{{{\Omega}}}(\cA_{\Omega}(z), \cA_{\Omega}(z^2), \ldots)$. The coefficients of all involved series are nonnegative, hence we may lift this identity of formal power series to an identity of real numbers. By assumption, $0 \in \Omega$ and there is an integer $\ell \ge 2$ such that $\ell \in \Omega$. Thus, for all $0 < x < {\rho_\Omega}$ 
	\begin{equation}
	\label{eq:AOmegalb}
	\cA_{\Omega}(x) \ge x\Big(1 + \frac{1}{\ell!} \sum_{\sigma \in \cS_\ell} \cA_{\Omega}(x)^{\sigma_1} \cA_{\Omega}(x^2)^{\sigma_2} \cdots \cA_{\Omega}(x^\ell)^{\sigma_\ell}\Big)
	\end{equation}
	with $\cS_\ell$ denoting the symmetric group of order $\ell$ and $\sigma_i$ denoting the number of cycles of length $i$ of $\sigma$. In particular, by considering the summand for $\sigma = \text{id}$, we have that
	\[
	\cA_{\Omega}(x) \ge x (\cA_{\Omega}(x))^{\ell} / \ell!.
	\]
	Since $\ell \ge 2$ this implies that the limit $\lim_{x \uparrow {\rho_\Omega}} \cA(x)$ is finite and hence $\cA_{\Omega}({\rho_\Omega})$ is finite.

	Moreover, considering the summand in~\eqref{eq:AOmegalb} for $\sigma$ a cycle of length $\ell$ yields that 
	\[
	\infty > \cA_{\Omega}({\rho_\Omega}) \ge {\rho_\Omega} (\cA_{\Omega}({\rho_\Omega}^\ell)) / \ell!.
	\]
	This implies that ${\rho_\Omega} \le 1$ because otherwise $\cA({\rho_\Omega}^\ell) = \infty$. If ${\rho_\Omega}=1$, then~\eqref{eq:AOmegalb} would imply that $\cA_{\Omega}(1) \ge 1$. Applying $(*)$ yields \[\cA_{\Omega}(1) \ge 1 + \cA_{\Omega}(1),\] which is clearly impossible. Hence our premise cannot hold and thus ${\rho_\Omega} < 1$.  We proceed with showing ii).
		We have that $\cA_{\Omega}(z) = E(z, \cA_{\Omega}(z))$. The series $E^{\Omega}(z,w)$ is dominated coefficient-wise by \[z\exp(w + \sum_{i\ge 2}\cA_{\Omega}(z^i)/i).\] Since ${\rho_\Omega} < 1$ it follows that there is an $\epsilon>0$ such that $E({\rho_\Omega} + \epsilon, \cA_{\Omega}({\rho_\Omega}) + \epsilon) < \infty.$ This establishes ii). To see the last claim, by a general enumeration result given in \cite[Thm. 28]{MR2240769} it follows that
		\[
		[z^m]\cA_{\Omega}(z) \sim \gcd(\Omega) \sqrt{\frac{{\rho_\Omega} E_z ({\rho_\Omega},\cA_{\Omega}({\rho_\Omega}))}{2\pi E_{ww}({\rho_\Omega},\cA_{\Omega}({\rho_\Omega}))}} {\rho_\Omega}^{-m} m^{-3/2}, \quad m \equiv 1 \mod \gcd(\Omega).
		\]
\end{proof}

\subsection{A Boltzmann sampler for random P\'olya trees}
Let $\Omega \subset \ndN_0$ denote a subset containing $0$ and at least one integer $\ge 2$. Recall that we let $\cA_\Omega$ denote the combinatorial class of P\'olya trees with vertex outdegrees in $\Omega$.

The size-preserving bijection in \eqref{eq:map} between the classes $\cA_\Omega$ and $\cX \cdot \MSet_\Omega(\cA_\Omega)$, where each tree corresponds to the multiset of trees pendling from its roots, allows us to construct a Boltzmann sampler $\Gamma \cA_\Omega$ for P\'olya trees.  The Boltzmann distribution is a measure on P\'olya trees with an arbitrary number of vertices. However, any tree with $n$ vertices has the same probability, i.e., the distribution conditioned on the event that the generated tree has $n$ vertices is \emph{uniform}. This will allow us to reduce the study of properties of a random P\'olya tree with exactly $n$ vertices to the study of $\Gamma \cA_{\Omega}$.

\begin{lemma}
	\label{SePole:sampler}
	The following recursive procedure~$\Gamma \cA_{\Omega}(x)$ terminates almost surely and draws a random P\'olya tree with outdegrees in~${\Omega}$ according to the \emph{Boltzmann distribution} with parameter~$0< x \le \rho_{\Omega}$, i.e. any object with~$n$ vertices gets drawn with probability~$x^n / \cA_{\Omega}(x)$.
	\begin{enumerate}[1.]
		\item Start with a root vertex~$v$.
		\item Let $\sigma(v)$ be a random permutation drawn from the union of permutation groups~$\bigcup_{k \in {\Omega}} \cS_k$ with distribution given by 
		\[
		\Pr{\sigma(v) = \nu}
		=
		\frac{x}{\cA_{\Omega}(x)} \frac{1}{k!} \cA_{\Omega}(x)^{\nu_1}\cA_{\Omega}(x^2)^{\nu_2} \cdots \cA_{\Omega}(x^k)^{\nu_k}
		\]
		for each~$k \in {\Omega}$ and~$\nu \in \cS_k$. Here~$\nu_i$ denotes the number of cycles of length~$i$ of the permutation~$\nu$. In particular,~$\nu_1$ is the number of fixpoints of~$\nu$.
		\item If~$\sigma(v) \in \cS_0$, then return the tree consisting of the root only and stop. Otherwise, for each cycle~$\tau$ of~$\sigma(v)$ let~$\ell_\tau \ge 1$ denote its length and draw a Polya tree~$A_\tau$ by an independent {\em recursive} call to the sampler~$\Gamma \cA_{\Omega}(x^{\ell_\tau})$. Make~$\ell_\tau$ identical copies of the tree~$A_\tau$ and connect their roots to the vertex~$v$ by adding edges. Return the resulting tree and stop.
	\end{enumerate}
\end{lemma}

A Boltzmann-sampler for $\Gamma \tilde{\cA}(x)$ is also explicitly described in \cite[Fig.\ 14, (1)]{MR2810913}. (Note that the exposition given there contains a typo, as it corresponds to attaching only one copy of each tree $A_\tau$ in Step $3$.)

We would like to justify the above procedure by applying the rules in Section~\ref{se:bosa} for obtaining samplers for products and multisets. Indeed, the product rule states, that a sampler $\Gamma \cA_\Omega(x)$ may be obtained by taking a root-vertex $v$ (which correspons to calling $\Gamma \cX(x)$), calling the multiset sampler $\Gamma \MSet_\Omega(\cA_\Omega)(x)$, and constructing a tree by connecting $v$ with the root-vertices of the obtained trees. The rule for multiset classes yields a procedure for $\Gamma \MSet_\Omega(\cA_\Omega)(x)$ that involves calls to $\Gamma \cA_\Omega(x^k)$ for several $k \ge 1$. If we interpret these calls as independent copies of Boltzmann distributed random variables, then the rules stated in Section~\ref{se:bosa} guarantee that the resulting random P\'olya tree follows a Boltzmann distribution with parameter $x$. However, this procedure is not "explicit", as we do not specify how to obtain these copies. Hence, instead, we interpret the indepent calls $\Gamma \cA_\Omega(x)$ as {\em recursive} calls to our constructed procedure, i.e. each call corresponds to again taking a root vertex and choosing (independently) multisets from $\MSet_\Omega(\cA_\Omega)$, which again may cause further recursive calls. This is similar to a branching process. Of course, we need to justify that this recursive procedure terminates almost surely and samples according to a Boltzmann distribution with parameter $x$. This justification is given in \cite[Thm. 4.2]{MR2498128} in a more general context for classes that may be recursively specified as in~\eqref{eq:map} using operations such as products and multiset classes. 

\subsection{Deviation Inequalities}
\label{sec:deviation}
We will make use of the following moderate deviation inequality for one-dimensional random walks found in most textbooks on the subject.

\begin{lemma}
\label{le:deviation}
Let $(X_i)_{i \in \ndN}$ be family of independent copies of a real-valued random variable $X$ with  $\Ex{X} = 0$. Let $S_n = X_1 + \ldots + X_n$. Suppose that there is a $\delta >0$ such that $\Ex{e^{\theta X}} < \infty$ for $|\theta| < \delta$. Then there is a $c>0$ such that for every $1/2 < p \le 1$ there is a number $N$ such that for all $n \ge N$ and $0 < \epsilon < 1$
\[
\Pr{|S_n/n^p| \ge \epsilon } \le 2 \exp(- c \epsilon^2 n^{2p-1}).
\] 
\end{lemma}

\section{Proof of the main theorem}\label{sec:pfmain}
{\em In the following~${\Omega}$ will always denote a set of nonnegative integers containing zero and at least one integer greater than or equal to two. Moreover,~$n$ will always denote a natural number that satisfies~$n \equiv 1 \mod \gcd({\Omega})$ and is large enough such that rooted trees with~$n$ vertices and outdegrees in~${\Omega}$ exist. 
}

\begin{proof}[Proof of Theorem~\ref{te:main}]

We begin the proof with a couple of auxiliary observations about the sampler $\Gamma \cA_{\Omega}(x)$ from Lemma~\ref{SePole:sampler}. Let us fix $x = \rho_{\Omega}$ throughout. We may do so, since by Proposition~\ref{SeFrpro:ser1} we have that $0 < \rho_{\Omega}<1$ and $\cA_{\Omega}(\rho_{\Omega}) < \infty$. 

Suppose that we modify Step 1 to "Start with a root vertex $v$. If the argument of the sampler is $\rho_{\Omega}$ (as opposed to $\rho_{\Omega}^i$ for some $i \ge 2$), then mark this vertex with the color blue.". Then the resulting tree is still Boltzmann-distributed, but comes with a colored subtree which we denote by $\cT$.

Note that~$\cT$ is distributed like a Galton-Watson tree without the ordering on the offspring sets. By construction, the offspring distribution~$\xi$ of~$\cT$ is given by the number of fixpoints of the random permutation drawn in Step 2. Thus, the probability generating function of~$\xi$ is 
\begin{align}
\label{SePoeq:a}
	\Ex{z^\xi}
	=
	\frac{\rho_{\Omega}}{\cA_{\Omega}(\rho_{\Omega})} Z_{{{\Omega}}}(z \cA_{\Omega}(\rho_{\Omega}), \cA_{\Omega}(\rho_{\Omega}^2), \cA_{\Omega}(\rho_{\Omega}^3), \ldots).
\end{align}
Moreover, for any blue vertex~$v$ we may consider the forest~$F(v)$ of the trees dangling from~$v$ that correspond to cycles of the permutation~$\sigma(v)$ with length at least two. Let~$\zeta$ denote a random variable that is distributed like the number of vertices~$|F(v)|$ in~$F(v)$. Then the probability generating function of~$\zeta$ is 
\begin{align}
\label{SePoeq:b}
	\Ex{z^{\zeta}}
	=
	\frac{\rho_{\Omega}}{\cA_{\Omega}(\rho_{\Omega})}
	Z_{{{\Omega}}}(\cA_{\Omega}(\rho_{\Omega}), \cA_{\Omega}((z\rho_{\Omega})^2), \cA_{\Omega}((z\rho_{\Omega})^3),\ldots).
\end{align}
Using Proposition~\ref{SeFrpro:cnt} it follows that the generating functions $\Ex{z^{\xi}}$ and $\Ex{z^{\zeta}}$ have radius of convergence strictly larger than one.
Hence~$\xi$ and~$\zeta$ have finite exponential moments. In particular, there are constants~$c,c' > 0$ such that for any~$s \ge 0$
\begin{equation}
\label{SePoeq:expTailBound}
\Pr{\xi \ge s} \le c e^{-c's} \quad \text{and} \quad \Pr{\zeta \ge s} \le c e^{-c's}.
\end{equation}
Moreover, as we argue below,~$\xi$ has average value
\begin{align*}
\label{SePoeq:c}
\Ex{\xi} = \left(\frac{\partial}{\partial s_1} Z_{{{\Omega}}}\right)(\cA_{\Omega}(\rho_{\Omega}), \cA_{\Omega}(\rho_{\Omega}^2), \ldots) \rho_{\Omega} = 1.
\end{align*}
This can be shown as follows. Recall that the ordinary generating series satisfies the identity~$\cA_{\Omega}(z) = E(z,\cA_{\Omega}(z))$ with the series~$E(z,w)$ given by
\[
E(z,w) = z Z_{{{\Omega}}}(w, \cA_{\Omega}(z^2), \cA_{\Omega}(z^3), \ldots).
\]
In particular, we have that~$F(z, \cA_{\Omega}(z)) = 0$ with~$F(z,w) = E(z,w) - w$. Suppose that $(\frac{\partial}{\partial w} F)(\rho, \cA_{\Omega}(\rho)) \ne 0$. Then by the implicit function theorem the function~$\cA_{\Omega}(z)$ has an analytic continuation in a neighbourhood of~$\rho_{\Omega}$. But this contradicts Pringsheim's theorem \cite[Thm. IV.6]{MR2483235}, which states that the series~$\cA_{\Omega}(z)$ must have a singularity at the point~$\rho_{\Omega}$ since all its coefficients are nonnegative real numbers. Hence we have~$(\frac{\partial}{\partial w} F)(\rho, \cA_{\Omega}(\rho)) = 0$ which is equivalent to~$\Ex{\xi}=1$.

With all these facts at hand we proceed with the proof of the theorem. Slightly abusing notation, we let $\mA_n$ denote the colored random tree drawn by conditioning the (modified) sampler $\Gamma \cA_{\Omega}(\rho_{\Omega})$ on having exactly~$n$ vertices. That is, if we ignore the colors, $\mA_n$ is drawn uniformly among all P\'olya trees of size $n$ with outdegrees in~${\Omega}$. Moreover, let $\cT_n$ denote the colored subtree of $\mA_n$, and for any vertex $v$ of $\cT_n$ let $F_n(v)$ denote the corresponding forest that consists of non-blue vertices.  We will argue that with high probability there is a constant $C > 0$ such that $|F_n(v)| \le C\log n$ for all $v \in \cT_n$. Indeed, note that by Proposition~\ref{SeFrpro:cnt},
\begin{equation}
\label{SePoeq:poly}
	\Pr{|\Gamma \cA_{\Omega}(\rho_{\Omega})| = n} = \frac{\rho_{\Omega}^n}{\cA_{\Omega}(\rho_{\Omega})}[z^n] \cA_{\Omega}(\rho_{\Omega}) = \Theta(n^{-3/2}),
\end{equation}
i.e. the probability is (only) polynomially small. Thus, for any~$s \ge 0$, if we denote by~$\zeta_1, \zeta_2, \dots$ independent random variables that are distributed like~$\zeta$
\[
\begin{split}
	\Pr{\exists v \in \cT_n: |F_n(v)| \ge s}
	& = \Pr{\exists v \in \cT: |F(v)| \ge s \mid |\Gamma \cA_{\Omega}(\rho_{\Omega})| = n} \\
	& \le O(n^{3/2}) \Pr{\exists 1 \le i \le n: \zeta_i \ge s}.
\end{split}
\]
Using~\eqref{SePoeq:expTailBound} and setting~$s = C\log n$ we get that~$\Pr{\zeta_i \ge s} = o(n^{-5/2})$ for an appropriate choice of~$C > 0$. Thus, by the union bound
\begin{equation}
\label{SePoeq:smallDangling}
	\Pr{\forall v \in \cT_n: |F_n(v)| \le C\log n} = 1 - o(1).
\end{equation}
\noindent
\begin{minipage}{.55\textwidth}
The typical shape of~$\mA_n$ thus consists of a colored tree with small forests attached to each of its vertices, compare with Figure \ref{SePofi:core}.  In particular, we have that the Gromov-Hausdorff distance between the rescaled trees~$\mA_n/\sqrt{n}$ and~$\cT_n/\sqrt{n}$ converges in probability to zero. We are going to show that there is a constant~$c_{\Omega}>0$ such that~$c_{\Omega} \cT_n/\sqrt{n}$ converges weakly towards the Brownian continuum random tree~$\CRT$. This immediately implies that~$$c_{\Omega} \mA_n/\sqrt{n} \convdis \CRT$$ and we are done.
\end{minipage}%
\begin{minipage}{.45\textwidth}
\centering
\includegraphics[width=0.53\textwidth]{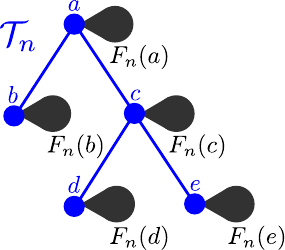}
\captionof{figure}{\small The typical shape of the random P\'olya tree with~$n$ vertices.}
\label{SePofi:core}
\end{minipage}

~\\

We are going to argue that the number of vertices in~$\cT_n$ concentrates around a constant multiple of~$n$. More precisely, we are going to show that for any exponent~$0<s<1/2$ we have with high probability that
\begin{align}
\label{SePoeq:d}
|\cT_n| \in (1 \pm n^{-s})\frac{n}{1 + \Ex{\zeta}}.
\end{align}
To this end, consider the corresponding complementary event in the unconditioned setting
\[
|\cT| \notin (1 \pm n^{-s})\frac{|\Gamma \cA_{\Omega}(\rho_{\Omega})|}{1 + \Ex{\zeta}}.
\]
If this occurs, then we clearly also have that
\[
\sum_{v \in \cT}(1 + |F(v)|) = |\Gamma \cA_{\Omega}(\rho_{\Omega})| \notin (1 \pm \Theta(n^{-s})) (1 + \Ex{\zeta}) |\cT|.
\]
Let~$\cE$ denote the corresponding event. From ~\eqref{SePoeq:smallDangling} we know that with high probability~$|F_n(v)| = O(\log n)$ for all vertices~$v$ of~$\cT_n$. Hence, with high probability, say,~$|\cT_n| \ge n/\log^2n$. Using again~\eqref{SePoeq:poly}
\[
\Pr{\cE \mid |\Gamma \cA_{\Omega}(\rho_{\Omega})|=n} = O(n^{3/2}) \Pr{  \frac{n}{\log^2n} \le |\cT| \le n, \cE} + o(1).
\]
By applying the union bound, the latter probability is at most
\[
\sum_{{n}/{\log^2n} \le \ell \le n} \Pr{ \sum_{i=1}^\ell (1 + \zeta_i) \notin (1 \pm \Theta(n^{-s}))  (1 + \Ex{\zeta})\ell}.
\]
Since the random variable~$\zeta$ has finite exponential moments, we may apply the deviation inequality in Lemma~\ref{le:deviation} in order to bound this by o(1).
Hence, \eqref{SePoeq:d} holds with probability tending to $1$ as $n$ becomes large. We are now going to prove that
\begin{align}
\label{SePoeq:e}
\frac{\sqrt{(1 + \Ex{\zeta})} \sigma}{2 \sqrt{n}}\cT_n \convdis \CRT 
\end{align}
with~$\sigma^2 > 0$ denoting the variance of the random variable~$\xi$. This implies that
\begin{equation}
\label{SePoeq:cOmega}
	c_{\Omega} \mA_n /\sqrt{n} \convdis \CRT
	\quad \text{ with } \quad
	c_{\Omega} = \frac{\sqrt{(1 + \Ex{\zeta})} \sigma}{2}
\end{equation}
and we are done. Note that~$\sigma$ and~$\Ex{\zeta}$ may be computed explicitly from the expression of the probability generating functions in~(\ref{SePoeq:a}) and (\ref{SePoeq:b}). We obtain that $\sigma^2$ is given by
\begin{align*}
	\sigma^2 &=  \left (\frac{\partial^2}{\partial z^2} \Ex{z^\xi} + \frac{\partial}{\partial z}\Ex{z^\xi}  - (\frac{\partial}{\partial z}\Ex{z^\xi})^2\right)(1) \\ 
	= &\quad{\rho_{\Omega}}{\cA_{\Omega}(\rho_{\Omega})}\frac{\partial^2 Z_{{{\Omega}}}}{\partial s_1^2}(\cA_{\Omega}(\rho_{\Omega}), \cA_{\Omega}(\rho_{\Omega}^2),\dots) \\ 
	&+ {\rho_{\Omega}}\frac{\partial Z_{{{\Omega}}}}{\partial s_1}(\cA_{\Omega}(\rho_{\Omega}), \cA_{\Omega}(\rho_{\Omega}^2),\dots) \\ 
	&- \rho_{\Omega}^2\left(\frac{\partial Z_{{{\Omega}}}}{\partial s_1} (\cA_{\Omega}(\rho_{\Omega}), \cA_{\Omega}(\rho_{\Omega}^2),\dots) \right)^2.
\end{align*}
Note that $\sigma > 0$, as $\xi$ is not constant. Moreover,
\begin{align*}
	\Ex{\zeta}
	&= \left(\frac{\partial}{\partial z}\Ex{z^\zeta}\right)(1) 
	= \frac{\rho_{\Omega}}{\cA_{\Omega}(\rho_{\Omega})}\sum_{i \ge 2}
		\left(\frac{\partial}{\partial s_i}Z_{{{\Omega}}}\right)( \cA_{\Omega}(\rho_{\Omega}),  \cA_{\Omega}(\rho_{\Omega}^2), \dots) \, i \rho_{\Omega}^i  \cA_{\Omega}'(\rho_{\Omega}^i),
\end{align*}
where $\cA_{\Omega}'(z) = \frac{\partial}{\partial z}\cA_{\Omega}(z)$. Note that this expression is well-defined, since $0 < \rho_{\Omega} < 1$.

In order to show~\eqref{SePoeq:cOmega}, let~$f: \ndK \to \ndR$ denote a bounded, Lipschitz-continous function defined on the space~$\ndK$ of isometry classes of compact metric spaces. Note that the tree~$\cT_n$ conditioned on having~$\ell$ vertices is distributed like the tree~$\cT$ conditioned on having~$\ell$ vertices. In particular, it is identically distributed to a~$\xi$-Galton-Watson tree~$\cT^\xi$ conditioned on having~$\ell$ vertices, which we denote by~$\cT^\xi_\ell$. Since (\ref{SePoeq:d}) holds with high probability it follows that
\[
\Ex{f({c_{\Omega} \cT_n}/{ \sqrt{n}})} = o(1) + \sum_{\ell \in (1 \pm n^{-s})\frac{n}{1 + \Ex{\zeta}}} \Ex{f({c_{\Omega} \cT^\xi_\ell}/{\sqrt{n}})} \Pr{|\cT| = \ell}.
\]
Let $\Di(T)$ denote the diameter of $T$, i.e., the number of vertices on a longest path in $T$. Since~$f$ was assumed to be Lipschitz-continuous it follows that 
\[
	\left|\Ex{f({c_{\Omega} \cT^\xi_\ell}/{\sqrt{n}})} - \Ex{f({\sigma \cT^\xi_\ell}/{2 \sqrt{\ell}})}\right|
	\le
	a_{n,\ell} \Ex{\Di(\cT^\xi_\ell)/\sqrt{\ell}}
\]
for a sequence~$a_{n,\ell}$ with~$\sup_\ell(a_{n,\ell}) \to 0$ as~$n$ becomes large. Moreover, the average rescaled diameter~$\Ex{\Di(\cT^\xi_\ell)/\sqrt{\ell}}$ converges to a multiple of the expected diameter of the CRT~$\CRT$ as~$\ell$ tends to infinity, see e.g.~\cite{MR3077536}. In particular, it is a bounded sequence. Since \[\Ex{f({\sigma \cT^\xi_\ell}/{2 \sqrt{\ell}})} \to \Ex{f(\CRT)}\] as~$\ell \to \infty$, it follows that \[\Ex{f({c_{\Omega} \cT_n}/{ \sqrt{n}})} \to \Ex{f(\CRT)}\] as~$n$ becomes large. This completes the proof.
\end{proof}

\begin{proof}[Proof of Theorem \ref{te:main2}]
We are going to use the notation of the previous proof. Let $x \ge 0$ be given. Without loss of generality, we may assume throughout that $x \ge \sqrt{n}$.  If the height $\He(\mA_n)$ of the tree $\mA_n$ satisfies $\He(\mA_n) \ge x$, then $\He(\cT_n) \ge x/2$ or $|F_n(v)| \ge x/2$ for at least one vertex $v \in \cT_n$. We are going to bound the probability for each of these events separately. By the tail bounds for conditioned Galton-Watson trees in Theorem~\ref{te:theoconv} there exist constants $C_1, c_1 > 0$ such that for all $\ell$ and $y\ge0$ we have that
\[
\Pr{\He(\cT) \ge y \mid |\cT| = \ell} \le C_1 \exp(-c_1 y^2/\ell).
\]
Moreover, $\cT_n$ conditioned on having size $\ell$ is distributed like $\cT$ conditioned on having size $\ell$. Hence
\begin{align}
	\label{eq:tail1}
\Pr{\He(\cT_n) \ge x/2} = \sum_{\ell=1}^n \Pr{|\cT_n|=\ell} \, \Pr{\He(\cT) \ge x/2 \mid |\cT| = \ell} \le C_1 \exp(-c_1 x^2/(4n)).
\end{align}
By~\eqref{SePoeq:poly} it holds that
\begin{align*}
\Pr{\max_{v \in \cT_n} |F_n(v)| \ge x/2} &\le O(n^{3/2}) \Pr{\max_{v \in \cT} |F(v)| \ge x/2, |\Gamma \cA_\Omega(\rho_\Omega)| = n} \\
&\le O(n^{5/2}) \Pr{\zeta \ge x/2}.
\end{align*}
As we assumed that $x \ge \sqrt{n}$, it follows by \eqref{SePoeq:poly} that there are constants $c,c',c''>0$ with
\[
n^{5/2} \Pr{\zeta \ge x/2} \le c \exp((5/2) \log n- c' x /2) \le c \exp(-c'' x).
\]
Consequently, there are constants $C_2, c_2 >0$ with 
\begin{align}
	\label{eq:tail2}
\Pr{\max_{v \in \cT_n} |F_n(v)| \ge x/2} \le C_2 \exp(-c_2 x^2/n).
\end{align}
Combining Inequalities \eqref{eq:tail1} and \eqref{eq:tail2} yields
\[
\Pr{\He(\mA_n) \ge x} \le C_1 \exp(-c_1 x^2/(4n)) + C_2 \exp(-c_2 x^2/n) \le C_3 \exp(-c_3 x^2/n)
\]
for some constants $C_3,c_3>0$. This concludes the proof of the tail-bound for the height.

In order to show the tail-bound for the width we begin with some auxiliary observations. Note that~\eqref{SePoeq:expTailBound} guarantees that there are $c \in (0,1)$ and $c' > 0$ such that for any $t,t' \in \mathbb{N}$
\[
	\Pr{\zeta \mathbf{1}(\zeta \ge t) \ge t'} \le c' c^{t'}
	\quad \text{ and } \quad
	\mathbb{E}\big[\zeta \mathbf{1}(\zeta \ge t) \big] \le c' t c^t,
\]
where $\mathbf{1}(E)$ denotes the indicator function for the event $E$. Let $\zeta_1, \zeta_2, \dots$ be independent random variables with the same distribution as $\zeta$ and define 
\begin{equation}
\label{eq:deftildezeta}
	\tilde \zeta = \sum_{i \ge 1} \zeta_i \mathbf{1}(\zeta \ge i).
\end{equation}
The previous observation implies that there is a $C > 0$ such that $\mathbb{E}[\tilde \zeta] \le C$ and $\tilde\zeta$ is finite almost surely.  Let $p_i = \Pr{\zeta < i}$ and for a  series $F(z)$ write $F^{\ge i}(z) = \sum_{j \ge i} [u^j]F(u) z^j$. Setting $F(z) = \mathbb{E}[z^\zeta]$, from~\eqref{SePoeq:b} we get that
\[
	\mathbb{E}[z^{\tilde \zeta}] = \prod_{i \ge 1} \left(p_i + F^{\ge i}(z) \right) = \prod_{i \ge 1} \left(1 + (p_i-1) + F^{\ge i}(z) \right).
\]
Note that for $z \ge 1$ we have that $(p_i-1) + F^{\ge i}(z) \ge 0$. Our assumption asserts that there is a $\rho > 1$ such that the radius of convergence of all $F^{\ge i}$'s equals $\rho$. Then, for any $z_0 \in (1, \rho)$ we have for some $A > 0$ and $a \in (0,1)$ that $F^{\ge i}(z_0) \le A a^i$ for all $i \in \mathbb{N}_0$. Thus
\[
	\sum_{i \ge 1} |(p_i-1) + F^{\ge i}(z_0)|
	\le \mathbb{E}[\zeta] + \sum_{i \ge 1}F^{\ge i}(z_0)
	< \infty.
\]
Since $\mathbb{E}[z^{\tilde \zeta}]$ has only non-negative coefficients, it follows that the radius of convergence of $\mathbb{E}[z^{\tilde \zeta}]$ is larger than one, and thus $\tilde \zeta$ has exponential moments as well. We thus infer that we even may choose $c \in (0,1)$ and $c' > 0$ such that for any $t' \in \mathbb{N}$
\begin{equation}
\label{eq:tailboundtildezeta}
	\Pr{\tilde \zeta \ge t'} \le c' c^{t'}.
\end{equation}
With all these facts at hand we proceed with the proof for the tail bound of the width. With foresight, set $\alpha = 1 + \mathbb{E}[\tilde \zeta]/2 > 0$. First, note that with Theorem~\ref{te:theoconv}, since $\cT_n$ is a conditioned $\xi$-Galton Watson tree with at most $n$ vertices
\begin{equation}
\label{eq:prwfirst}
	\Pr{\Wi(\mA_n) \ge x}
	\le \Pr{\Wi(\mA_n) \ge x \text{ and } \Wi(\cT_n) \le x/\alpha} + C_4\exp(-c_4x^2/n)
\end{equation}
for some $c_4, C_4 > 0$; here and in the sequel we identify $\mA_n$ with  $\Gamma \cA_\Omega(\rho_\Omega)$ conditional on $|\Gamma \cA_\Omega(\rho_\Omega)| = n$, and $\cT_n$ is its colored subtree. Let $L_d$ be the set of vertices in $\cT_n$ with distance $d$ from the root, and define $W_d = |L_d|$. Conditional on ``$\Wi(\cT_n) \le x/\alpha$'' we have that $W_i \le x/\alpha$ for all $1 \le i \le n$ (and, of course, $W_i = 0$ for all other $n$). Now note that the number of vertices at distance $d$ from the root of $\mA_n$ is bounded from above by
\[
	U_d := W_d + \sum_{i = 0}^{d-1} \sum_{v \in L_i} |F_n(v)| \mathbf{1}(|F_n(v)| \ge d-i),
\]
since a forest $F_n(v)$ with $v \in L_i$ cannot contribute to the set of depth $d$ vertices in $\mA_n$ unless it has at least that $d-i$ vertices. Thus
\[
	\Pr{\Wi(\mA_n) \ge x \text{ and } \Wi(\cT_n) \le x/\alpha}
	\le \Pr{\exists 1 \le d \le n: U_d \ge x \text{ and } \Wi(\cT_n) \le x/\alpha}.
\]
Using the union bound, ~\eqref{SePoeq:poly} and the definition in~\eqref{eq:deftildezeta}, we infer that there is a $C_5 > 0$ such that the latter probability is bounded by
\[
	C_5n^{5/2} \Pr{\sum_{i = 1}^{x/\alpha} \tilde \zeta_i \ge (1 - 1/\alpha)x}, \quad\text{where } \tilde \zeta_1, \tilde \zeta_2, \dots \text{ are iid with the same distribution as } \tilde\zeta.
\]
Note that with the choice of $\alpha$ we have that the expectation of the sum equals $(1 - 1/\alpha)x/2$. Since $\tilde \zeta$ has exponential moments, see~\eqref{eq:tailboundtildezeta}, by applying Lemma~\ref{le:deviation} we infer that there are constants $c_6, C_6 > 0$ such that
\[
	\Pr{\exists 1 \le d \le n: U_d \ge x \text{ and } \Wi(\cT_n) \le x/\alpha}
	\le C_6 e^{-c_6 x},
\]
and the proof is completed with~\eqref{eq:prwfirst}.
\end{proof}

\bibliographystyle{plain}
\bibliography{polya}

\end{document}